\RequirePackage{ifpdf}
\ifpdf
 	\documentclass[pdftex]{amsart}
 	\else
 	\documentclass[dvips]{amsart}
\fi

\usepackage{amsfonts,amsthm,latexsym,amsmath,amssymb,amscd,amsmath, mathrsfs, epsf}
\usepackage{graphicx,xcolor}

\graphicspath{{./figures/}}
\ifpdf
	\usepackage{epstopdf}		
	\DeclareGraphicsExtensions{.png,.jpg,.eps,.epsf}
\fi

\usepackage{hyperref}

\newtheorem{theorem}{Theorem}
\newtheorem{proposition}[theorem]{Proposition}

\newtheorem{definition}{Definition}
\newtheorem{corollary}{Corollary}
\newtheorem{example}{Example}
\newtheorem{remark}{Remark}
\newtheorem{conjecture}{Conjecture}

\newtheorem{problem}{Problem}

 \newcommand{\al}{\alpha}

 \newcommand{\bC}{\mathbb{C}}

\newcommand{\cD}{\mathcal{D}}	

\newcommand {\ldr} {\langle t \rangle}

\begin{document}
\title[Secant degeneracy index of  strata of binary forms]{Secant degeneracy index of the standard strata in the space of binary forms}
\author[G.~Nenashev]{Gleb Nenashev}

\address{ Department of Mathematics,
   Stockholm University,
   S-10691, Stockholm, Sweden}
\email{nenashev@math.su.se}

\author[B.~Shapiro]{Boris Shapiro}

\address{Department of Mathematics,
   Stockholm University,
   S-10691, Stockholm, Sweden}
\email{shapiro@math.su.se}

\author[M.~Shapiro]{Michael Shapiro}

\address{Department of Mathematics, Michigan State University,
          East Lansing, MI 48824-1027}
\email{mshapiro@math.msu.edu}

\begin{abstract}
The  space $Pol_d\simeq \bC P^d$ of all complex-valued binary forms of degree $d$ (considered up to a constant factor) has a standard stratification, each   stratum of which   contains all forms whose set of multiplicities of their  distinct roots is given by a fixed partition $\mu \vdash d$. For each such stratum $S_\mu,$ we introduce its secant degeneracy index   $\ell_\mu$ which is the minimal number of projectively dependent pairwise distinct points on  $S_\mu$, i.e., points  whose projective span has  dimension smaller than $\ell_\mu-1$. In what follows, we 
discuss the secant degeneracy index $\ell_\mu$ and the secant degeneracy index $\ell_{\bar \mu}$ of the closure $\bar S_\mu$. 
\end{abstract}




\maketitle

\section{Introduction} 
Below by a \emph{form} we will always mean a \emph{binary form}. 
The standard stratification of the $d$-dimensional projective space $Pol_d$ of all complex-valued binary forms of degree $d$ (considered up to a non-vanishing constant factor) according to the multiplicities of their distinct roots  is a well-known  and widely used construction in mathematics, see e.g. \cite{Ar,Va, KhSh}. Its strata  denoted by $S_\mu$ are enumerated by  partitions  $\mu\vdash d$. In particular, cohomology of $S_\mu$ with different coefficients appears in many topological problems and  were intensively studied over the years, see e.g. \cite{Va} and references therein. 

\begin{definition} {\rm Given a positive-dimensional quasi-projective variety $V\subset \bC P^d$, we define its \emph{ secant degeneracy index} $\ell_V$  as the minimal positive integer $\ell$ such that there exists $\ell$ distinct  points  on  $V$ which are projectively dependent, i.e. whose projective span has dimension at most $\ell-2$. (Observe that singular points of $V$ are not considered as collapsing distinct points. For example, according to our definition, a point of self-intersection of $V$ is still considered a just one point of $V$.)}
\end{definition}

\begin{remark}{\rm   
The secant degeneracy index in much more general context has been introduced by e.g. Beltrametti and Sommese in~\cite{BS} while discussing the $\ell$-ampleness of the linear system of hyperplane sections for different $\ell$.  This notion was further developed in a recent paper \cite{CC}.  (We have to mention that unlike many modern authors, we only  consider reduced finite subschemes of~$V$.) As was pointed out to us by the anonymous referee, the secant degeneracy index has already appeared in a number of topics in algebraic geometry and, in particular, is closely connected with the identifiability of tensors and higher order  normality. For example, a related question about the uniqueness  of representation of generic forms of subgeneric rank as sums of powers of linear forms has been studied in a recent article \cite{COV}.}  
\end{remark}

\begin{remark}{\rm 
Obviously, $3\le\ell_V\le d+2$. The upper bound is attained  for a  rational normal curve in $\bC P^d$.   On the other hand, if $V$ contains $\bC P^1 \setminus \{\text{finite set}\}$, then $\ell_V=3$. We owe to  the anonymous referee the important observation  that the above trivial upper bound  $\ell_V=d+2$  is attained only on (Zariski open subsets of) rational normal curves and this bound can  be improved as follows. Namely, if $c:=\text{codim\;} V=d-\dim V$, then $\ell_V\le c+2$ unless $V$ is a variety of minimal degree which in our notation means that $\deg V=c+1$. By the classical results of Del Pezzo and Bertini, every positive-dimensional variety of minimal degree contains a line which implies that  $\ell_V=3$,  unless $V$ is a rational normal curve or a Veronese surface, see e.g. Theorem 1 of \cite {EiHa}. For the Veronese surface (which is a surface in $\bC P^5$), $\ell_V=4=c+1.$ To see that, one can take $4$ points on a conic. Thus, one conclude that $\ell_V\le c+2$ unless $V$ is a Zariski open subset of a rational normal curve in which case $\ell_V=c+3=d+2$.
}
\end{remark}

\medskip
Observe that if a quasi-projective variety $V$ is contained in quasi-projective $W$, then $\ell_V\ge \ell_W$. For a positive-dimensional quasi-projective variety $V\subset \bC P^d$, denote by $\ell_{\bar V}$ the secant degeneracy index of the closure $\bar V \subset \bC P^d$. Obviously, $\ell_{\bar V} \le \ell_{V}$. The latter inequality can be strict as shown by Example~\ref{ex} below.


\medskip 
The principal question considered in the present paper is as follows.

\begin{problem}\label{prob:index}
For a given partition $\mu\vdash d$, calculate/estimate  its secant degeneracy  indices $\ell_\mu:=\ell_{S_\mu}$ and  $\ell_{\bar \mu}:=\ell_{\bar S_\mu}$. 
\end{problem}

\medskip

\begin{example}\label{ex}
{\rm For 
$\mu(d)=(2d+1,d,d,d,d)$ and $\mu'(d)=(2d+1,2d,2d)$, $\ell_{\bar \mu(d)}=\ell_{\mu'(d)}=3$, but $\ell_{\mu(d)}$ grows to infinity when $d\to \infty$. (This result follows from Theorem~\ref{minjump} below.)}
\end{example}

For a given partition $\mu$, the equation 
\begin{equation}\label{eq:SFMU}
f_{1}+f_2+\dots +f_{ \ell_{\mu}}=0,
\end{equation} 
is called the {\it minimal secant degeneracy relation}   for $S_\mu$. A solution of the latter equation is a collection of pairwise non-proportional forms from $S_\mu$ satisfying \eqref{eq:SFMU}. 

Analogously, for a given partition $\mu$, the equation 
\begin{equation}\label{eq:SFMUb}
f_{1}+f_2+\dots +f_{ \ell_{\bar \mu}}=0,
\end{equation} 
is called the {\it minimal secant degeneracy relation for }   $\bar S_\mu$. A solution of the latter equation is a collection of pairwise non-proportional forms from $\bar S_{ \mu}$ satisfying \eqref{eq:SFMUb}. 

\medskip

Most of our results  deal with the secant degeneracy  index $\ell_{\mu}$. However,  the second part of Theorem~\ref{minjump}  provides a non-trivial lower bound for $\ell_{\bar \mu}$ generalizing a similar result of a well-known paper \cite{NeSl} from 1979 where the special case of  partitions with equal parts was considered. 

\medskip
 The first result of this note is as follows.  
Recall the notion of the refinement  partial order $``\succ"$ on the set of all partitions of a given positive integer  $d$. Namely, $\mu^\prime \succ \mu$ in this order if $\mu^\prime$ is obtained from $\mu$ by merging of some parts of $\mu$.  The unique minimal element of this partial order is $(1)^d,$ while its unique maximal element is $(d)$. 

\medskip
For a partition $\mu=(\mu_1\ge \mu_2 \ge \dots \ge \mu_r)$,  define its \emph{jump multiset}  
$J_\mu$ as the multiset of all positive  numbers in the set  $\{\mu_1-\mu_2,\ldots,\mu_{r-1}-\mu_r, \mu_r\}$. We denote by  $h_\mu$ the minimal  (positive) jump of $\mu$, i.e. the minimal element of $J_\mu$, and by $h_{\bar{\mu}}$ the minimal jump of all partitions $\mu'\succeq  \mu$.


\begin{theorem}
\label{minjump}
For any $\mu=(\mu_1\ge \mu_2 \ge \dots \ge \mu_r)$,  
$$\rm{(i)} \quad \quad\ell_\mu > \sqrt{h_\mu+\frac{1}{4}}+\frac{3}{2},$$
and
$$\rm{(ii)} \quad \quad \ell_{\bar{\mu}}> \sqrt{h_{\bar{\mu}}+\frac{1}{4}}+\frac{3}{2}.$$
\end{theorem}

\medskip
To formulate further results, we divide the set of all partitions into two natural disjoint subclasses as follows. 

\medskip
\noindent
{\bf Notation.} 
For a given partition $\mu=(\mu_1\ge \mu_2 \ge \dots \ge \mu_r)$ and a non-negative integer $t$, define   the partition $\mu^{\ldr}$ as $$\mu^{\ldr} :=(\mu_1+t\ge \mu_2+t\ge \dots \ge \mu_r+t).$$

\begin{definition}
{\rm We say that a partition $\mu$ has a \emph {growing secant degeneracy index} if $\lim_{t\to \infty}\ell_{\mu^{\ldr}}=+\infty$ and we say that $\mu$ has a \emph {stabilising secant degeneracy index} otherwise. }�
\end{definition}

We are able to characterize these two classes in the following terms. 

\begin{definition}
{\rm Given a partition $\mu$ and a positive integer $m$,  a solution of 
\begin{equation}\label{eq:SFGEN}
f_1+f_2+\dots +f_m=0,
\end{equation} 
with pairwise non-proportional $f_i\in S_\mu$ is called 
 a \emph{common radical solution} if all $f_i$'s have the same radical, i.e. the same set of distinct linear factors (considered up to a constant factor). We call a partition $\mu$ such there exists $m$ and a common radical solution of \eqref{eq:SFGEN} a \emph{partition admitting a common radical solution}. } 
\end{definition}

\medskip
The following proposition is straightforward. 
\begin{proposition}\label{pr:general}
A partition $\mu=(\mu_1\ge \mu_2\ge \dots \ge \mu_r)=(i_1^{m_1}, i_2^{m_2},\dots, i_s^{m_s})$ with  distinct $i_j$'s has a stabilising secant degeneracy index if and only if, for some positive integer $m,$ there exists a common radical solution of \eqref{eq:SFGEN}.  A partition $\mu$ as above has a growing secant degeneracy index if and only if  the linear span of the $Sym_r$-orbit of any form $f\in S_\mu$ has the dimension equal to the multinomial coefficient $\frac{r!}{m_1!m_2!\dots m_s!}$.  (Here the symmetric group $Sym_r$ acts on any $f\in S_\mu$ by permuting all its $r$ distinct roots.)
\end{proposition}

At the moment we do not have a purely combinatorial description of partitions admitting a  common radical solution. However we were able to study  a somewhat stronger property. 

\begin{definition}
{\rm We say that a partition $\mu=(\mu_1\ge \mu_2\ge \dots \ge \mu_r)$ admits  a  \emph{strongly common radical solution} if for some integer $m$,  a common radical solution exists for any choice of $r$ distinct roots.}
\end{definition}


\begin{theorem}\label{prop:parking}
A partition $\mu=(\mu_1 \ge \mu_2 \ge \dots \ge \mu_r)$ admits a strongly common radical solution if  there exists  a sequence $\{a_1,\ldots,a_r\}$ of positive integers such the number of different permutations $\pi$ of $\mu$ such that $(\pi \circ \mu)_i\geq {a_i}$ is at least $|\mu|-\sum_{i=1}^r{a_i}+2$, where  $(\pi \circ \mu)_i$ is the $i$-th entry of the partition $(\pi \circ \mu).$

\end{theorem}

In fact, we strongly suspect that the converse to Theorem~\ref{prop:parking} holds as well. 

\begin{conjecture}\label{conj:converse}  A necessary and sufficient condition for a partition $\mu=(\mu_1 \ge \mu_2 \ge \dots \ge \mu_r)$ to admit a strongly common radical solution is given by the existence of  a sequence $\{a_1,\ldots,a_r\}$ of positive integers such the number of different permutations $\pi$ of $\mu$ such that $(\pi \circ \mu)_i\geq {a_i}$ is at least $|\mu|-\sum_{i=1}^r{a_i}+2$, where  $(\pi \circ \mu)_i$ is the $i$-th entry of the partition $(\pi \circ \mu)$.
\end{conjecture} 

At the moment we can settle Conjecture~\ref{conj:converse} for a large class of partitions, but not for all partitions.

\medskip
The structure of the paper is as follows. In Section~\ref{sec:index}, we formulate several general results about $\ell_{\mu}$, the most interesting of them being an  upper bound of $\ell_\mu$ in terms of the  minimal jump.   
In Section~\ref{sec:gr&st}, we discuss common radical solutions  of \eqref{eq:SFMU} and  
in Section~\ref{finsec}, we present a number of open problems. 


\bigskip
 
\noindent
{\bf Acknowledgements.}  The second author is grateful to Professor B.~Reznick of UIUC for  discussions of the topic. The third author wants to acknowledge the hospitality of the Mathematics Department, Stockholm University in September-October 2015. The authors want to express their gratitude to the anonymous referee whose constructive criticism allowed us to substantially improve the quality of the exposition.

\medskip
\section {General results on the  secant degeneracy index}\label{sec:index}

 Given a partition  $\mu=(\mu_1\ge \mu_2 \ge \dots \ge \mu_r)$, we call $\nu=(\mu_{i_1}\ge \mu_{i_2}\ge \dots \ge \mu_{i_s})$ where $1\le i_1<i_2<\dots <i_s\le r$, a {\it subpartition} of $\mu$. 

\begin{proposition}\label{prop2} For a partition  $\mu=(\mu_1\ge \mu_2 \ge \dots \ge \mu_r)$ and any   subpartition $\nu$ of $\mu$, the inequality $$\ell_\mu\le \ell_{\nu}$$
holds.  In particular, $\ell_\mu\le \mu_r+2$.\end{proposition}

\begin{proof} Given a subpartition $\nu=(\mu_{i_1}\ge \mu_{i_2}\ge \dots \ge \mu_{i_s})$ of a partition $\mu=(\mu_1\ge \mu_2 \ge \dots \ge \mu_r)$, let 
\begin{equation*}
f_{1}+f_2+\dots +f_{ \ell_{\nu}}=0,
\end{equation*} 
be a linear dependence of pairwise non-proportional  binary forms from $S_{\nu}$ realizing its secant degeneracy index. Take the partition $\widehat \mu=\mu \setminus \nu=(\widehat \mu_1\ge \widehat \mu_2 \ge  \dots \ge \widehat \mu_{r-s})$. Multiplying the latter equality by $\prod_{j=1}^{r-s} (x-a_jy)^{\widehat \mu_j},$ where $a_j$ are generic complex numbers, we get a linear dependence between polynomials in $S_\mu$. 
The inequality  $\ell_\mu\le \mu_r+2$ is a special case of the general inequality, if one chooses $\nu=(\mu_r)$.   
Observe  that for the partition $(d)\vdash d$, $\ell_{(d)}=d+2$,  since the set of   binary forms of degree $d$ with a root of multiplicity $d$ is a rational normal curve in $Pol_d\simeq \bC P^d$. 
\end{proof} 

\begin{example}{\rm 
The latter upper bound $\ell_\mu\le \mu_r+2$ is sharp in case of any partition with $\mu_r=1,$ but not in general. Namely,  
already for $\mu=(2^2) \vdash 4$, $\ell_\mu=3 < 4=\mu_2+2$.  For $\mu=(3^2)\vdash 6$, $\ell_\mu=4<5$; 
$\mu=(4^2)\vdash 8$, $\ell_\mu=4<6$, see \cite{Re}. }
\end{example}

\medskip
Before formulating general results about $\ell_\mu$, let us present several concrete classes of $\mu$ and some information about the corresponding $\ell_\mu$.

\begin{proposition}  
Let $\mu=(\mu_1\ge \mu_2 \ge \dots \ge \mu_r)$ be a partition with two different indices $i_1$ and $i_2$ such that $\mu_{i_1}-\mu_{i_1+1}=\mu_{i_2}-\mu_{i_2+1}=1$. Then, $\ell_\mu\le 4$.
\end{proposition}  

\begin{proof}
Without loss of generality, assume that $i_1<i_2$, and consider two different cases.

\smallskip
\noindent
Case 1. $i_2=i_1+1.$ Take a subpartition $\nu=(\mu_{i_1}, \mu_{i_1+1},\mu_{i_1+2})=(\mu_{i_1+2}+2, \mu_{i_1+2}+1,\mu_{i_1+2})$ and set $k=\mu_{i_1+2}$.  We know that $\ell_\mu\leq \ell_{\nu}$. So it is enough to prove that $\ell_{\nu}\leq 4$. Take three distinct complex numbers $p,q$ and $r$, and consider four polynomials 
$$g_1=(x-p)^{k+2}(x-q)^{k+1}(x-r)^{k}, \quad g_2=(x-p)^{k+2}(x-r)^{k+1}(x-q)^{k},$$ 
$$g_3=(x-q)^{k+2}(x-p)^{k+1}(x-r)^{k},\quad g_4=(x-r)^{k+2}(x-p)^{k+1}(x-q)^{k}.$$
A  linear combination $ag_1+bg_2+cg_3+dg_4$ is given by 
$$Q(x)(a(x-p)(x-q)+b(x-p)(x-r)+c(x-q)^2+d(x-r)^2),$$
where $Q(x)=(x-p)^{k+1}(x-q)^k(x-r)^k$. 
Polynomials $(x-p)(x-q)$, $(x-p)(x-r)$, $(x-q)^2$ and $(x-r)^2$ are linearly dependent. Thus  there exist  $a,b,c,d$ such that $ag_1+bg_2+cg_3+dg_4=0$. Hence 
$\ell_{\nu}\leq 4$.

\smallskip
\noindent
Case 2. $i_2>i_1+1$. Take a subpartition $\nu=(\mu_{i_1}, \mu_{i_1+1},\mu_{i_2}, \mu_{i_2+1})=(\mu_{i_1+1}+1, \mu_{i_1+1},\mu_{i_2+1}+1, \mu_{i_2+1})$; set $k_1=\mu_{i_1+1}$ and $k_1=\mu_{i_2+1}$.  We know that $\ell_\mu\leq \ell_{\nu}$. So it is enough to prove that $\ell_{\nu}\leq 4$. Take four distinct complex numbers $p,q,r$ and $t$, and consider four polynomials 
$$g_1=(x-p)^{k_1+1}(x-q)^{k-1}(x-r)^{k_2+1}(x-s)^{k_2}, \quad g_2=(x-q)^{k_1+1}(x-p)^{k-1}(x-r)^{k_2+1}(x-s)^{k_2},$$
 $$g_3=(x-p)^{k_1+1}(x-q)^{k-1}(x-s)^{k_2+1}(x-r)^{k_2},\quad g_4=(x-q)^{k_1+1}(x-p)^{k-1}(x-s)^{k_2+1}(x-r)^{k_2}.$$

\medskip
A linear combination $ag_1+bg_2+cg_3+dg_4$ is given by 
$$R(x)(a(x-p)(x-r)+b(x-q)(x-r)+c(x-p)(x-s)+d(x-q)(x-s)),$$
where $R(x)=(x-p)^{k_1}(x-q)^{k_1}(x-r)^{k_2}(x-r)^{k_2}$. 
Polynomials $(x-p)(x-r)$, $(x-q)(x-r)$, $(x-p)(x-s)$ and $(x-q)(x-s)$ are linearly dependent. Thus there exist $a,b,c,d$ such that $ag_1+bg_2+cg_3+dg_4=0$, and hence 
$\ell_{\nu}\leq 4$.
\end{proof}

\begin{definition}
{\rm By the {\it radical} of a given binary form we mean the binary form obtained as the  product of all distinct linear factors of the original form.} 
\end{definition} 

\begin{proposition}  For any partition $\mu= (\mu_1\ge \mu_2 \ge \dots \ge \mu_r)$ and  given an arbitrary positive integer $i$, consider the partition $\mu^\prime=(\mu_1+i \ge \mu_2+i \ge \dots \ge  \mu_r+i\ge i ,i,\dots, i )$, where the entry $i$  is repeated $r(\ell_{ \mu} -1)$ times at the end of $\mu^\prime$. Then, $\ell_{\mu^\prime} \le \ell_{\mu}$.
\end{proposition}

\begin{proof}
Let $f_1,\dots,f_{\ell_\mu}$ be a solution of~\eqref{eq:SFMU}. 
Consider the radical $g$ of the polynomial $f_1  f_2 \dots f_{\ell_\mu}$. Since any form $f_j$ has exactly $r$ distinct roots, the degree of $g$ is  less than $r \ell_\mu$. 

Construct $g'$ as the product of $g$ by  $r \ell_\mu-deg(g)$ new distinct linear forms, and set  $f'_j = f_j \cdot (g')^i$, for  $j=1,\dots,  \ell_\mu$. 
It is easy to see that each $f'_j$ has the root partition given  by $\mu'$. 
Furthermore, one has $$f'_1+\dots+f'_{\ell_\mu}=(f_1+\dots+f_{\ell_\mu})\cdot (g')^i=0,$$
hence, $\ell_{\mu^\prime} \le \ell_{\mu}$.
\end{proof} 

\begin{corollary}
For any partition $\mu$ containing  the subpartition  $\nu = (t+1, t,t),$  where $t$ is any  positive integer, the  secant degeneracy index $\ell_\mu$ equals $3$. 
 More generally,  for   any  positive integer $t$, and any partition $\mu$ containing  the subpartition  $\nu = (t+i, \underbrace{t,t,\dots, t}_{i+1}),$     the secant degeneracy index $\ell_\mu$ is at most  $i+2$. 
\end{corollary}

\medskip We finish this section with the proof of Theorem~\ref{minjump}.

\begin{proof}
 Both parts of Theorem~\ref{minjump} are settled in a similar way. Namely, 
given $\mu$, let $\{f_1,\ldots,f_{\ell}\}$   be a collection of forms solving either \eqref {eq:SFMU}   or \eqref {eq:SFMUb}. (In the first case $\ell=\ell_\mu$ and in the second case $\ell=\ell_{\bar \mu}$.) Assume that  $\{f_1,\ldots,f_{\ell}\}$ gives a counterexample to the statement. 
Denote by  $g$ the GCD of $\{f_1,\ldots,f_{\ell}\}$ and 
consider the relation $$\frac{f_1}{g}+\ldots+\frac{f_{\ell}}{g}=0.$$

In case (i) of Theorem~\ref{minjump}, for any $i$, every root of the polynomial $\frac{f_i}{g}$ has  multiplicity at least than $h_\mu$, because this multiplicity  equals $\mu_k-\mu_l$ for some $k\leq l$.  Observe that, for  all $k\le l$, $\mu_k-\mu_l$ is either $0$ or is greater than  or equal to $h_\mu$.  

In case (ii) of Theorem~\ref{minjump}, for any $i$, every root of the polynomial $\frac{f_i}{g}$ has  multiplicity not smaller than $h_{\bar \mu}$, because this multiplicity  equals $\sum_{k\in A} \mu_k-\sum_{l\in B}\mu_l$,  where $A$ and $B$ are two subsets of $\{1,\dots , r\}$.  Observe that, for  all pairs $(A',B')$ such that $A'\cap B'=\emptyset$, $|\sum_{k\in A'} \mu_k-\sum_{l\in B'}\mu_l|$  is either $0$ or is greater than  or equal to $h_{\bar\mu}$.  

\medskip
The rest of the proof is the same in both cases. In what follows, $h$ stands for $h_\mu$ in case (i) and for $h_{\bar \mu}$ in case (ii). 
Consider the sequence of Wronskians $$w_i=W\left(\frac{f_1}{g}, \ldots,\frac{f_{i-1}}{g},\frac{f_{i+1}}{g},\ldots,\frac{f_{\ell}}{g}\right),\; i=1,\dots, \ell.$$ 
 All these Wronskians are proportional to each other due to the latter relation. 

Let $\alpha$ be a  root of some $f_i$. There exists an index  $s$ such that $\frac{f_s}{g}$ is not divisible by $(x-\alpha)$, since otherwise $g$ is not the GCD. 

For any $t,$ consider the multiplicity of the root of $w_t$ at $\al$.  It satisfies the inequality:  
$$ord_\alpha (w_t)\ge \sum\left(ord_\alpha\left(\frac{f_j}{g}\right)\right)-(\ell-2)\#\left\{i:(x-\alpha)|\frac{f_i}{g}\right\},$$ 
because any column of the  Wronski matrix corresponding to $(x-\alpha)|\frac{f_j}{g}$ is divisible by $(x-\alpha)^{ord_\alpha{\left(\frac{f_j}{g}\right)}-\ell+2}$.

\medskip
Hence, $$\deg w_1  \ge \sum_{i=1}^{\ell} \left(\deg \left(\frac{f_i}{g}\right)-(\ell-2)\#_{roots}\left(\frac{f_i}{g}\right)\right)=$$ $$=\ell (|\mu|- \deg g )-(\ell-2)\sum_{i=1}^{\ell} \#_{roots}\left(\frac{f_i}{g}\right).$$
On the other hand, $$\deg w_1 \leq (\ell - 1)\left(\deg\left(\frac{f_i}{g}\right)-\ell +2\right)=(\ell-1) (|\mu|- \deg g)-(\ell-1)(\ell-2).$$
We obtain 
$$(\ell-1) (|\mu|- \deg g)-(\ell-1)(\ell-2)\ge \ell (|\mu|- \deg g)-(\ell-2)\sum_{i=1}^{\ell} \#_{roots}\left(\frac{f_i}{g}\right),$$
i.e., 
$$(\ell-2)\sum_{i=1}^{\ell} \#_{roots}\left(\frac{f_i}{g}\right)-(\ell-1)(\ell-2)\ge |\mu|- \deg g.$$

The number $\#_{roots}\left(\frac{f_i}{g}\right)$ of distinct roots  is at most 
$\frac{|\mu|- \deg g}{h}$, because each root has  multiplicity at least ${h}$. Thus 
 $$(\ell-2)(\ell-1)\frac{|\mu|- \deg g}{h_\mu}-(\ell-1)(\ell-2)\ge |\mu|- \deg g.$$
 Hence, $(\ell-2)(\ell-1)>{h}.$  
\end{proof}

\medskip
\section{Partitions with growing and  stabilising secant degeneracy index}
\label{sec:gr&st}

\begin{theorem}
\label{onlycomrad}
 If  $\mu=(\mu_1\ge \mu_2 \ge \dots \ge \mu_r)$ satisfies the inequality 
 $$m\leq \sqrt{\frac{\mu_r}{r-1}}+1,$$ then any solution of \eqref{eq:SFGEN} is a common radical solution.
\end{theorem}

\begin{proof}
Assume the opposite. Let $\{f_1,\ldots,f_{m}\}$ be  a solution of \eqref{eq:SFMU} which is not a common radical solution. 
Let  $g$ be the GCD of $\{f_1,\ldots,f_{m}\}$.

For the term $f_i=c_i(x-a_{i,1})^{\mu_1}\cdots (x-a_{i,r})^{\mu_r}$, define 
$$g_i:=(x-a_{i,1})^{\mu_1-m+2}\cdots (x-a_{i,r})^{\mu_r-m+2}.$$
  Observe that $g_i$ is a polynomial, because any root of $f_i$ has  multiplicity at least $\mu_r>m$.

Consider the sequence of Wronskians 
$$w_i=W(f_1, \ldots,f_{i-1},f_{i+1},\ldots,f_{m}),\; i=1,\dots , m.$$
 They are proportional to each other, because $f_1+\ldots+f_m=0$. 
Notice that, for $i\neq t,$ the column in the Wronski matrix for $w_t$ corresponding to $f_i$  is divisible by $g_i$.  
Hence $w_t$ is divisible by ${\prod_{i=1}^{m} g_i}/{g_t}$.

\medskip
Since $\{f_1,\ldots,f_{m}\}$ is not a common radical solution,  there exists $\alpha\in \bC$, such that $\alpha$ is a root of $f_{p}$ but  is not a root of $f_q$ for some $p\neq q$.
  
Since the Wronskians $w_p$ and $w_q$  are proportional,  they are divisible by 
$$LCM\left(\frac{\prod_{i=1}^{m} g_i}{g_p},\frac{\prod_{i=1}^{m} g_i}{g_q}\right)=\frac{\prod_{i=1}^{m} g_i}{GCD(g_p,g_q)}=\frac{\prod_{i=1}^{m} g_i}{g_p}\frac{g_p}{GCD(g_p,g_q)}.$$
Then these Wronskians are divisible by $\frac{\prod_{i=1}^{m} g_i}{g_p}(x-\alpha)^{\mu_r-m+2}$. Therefore  their degrees are greater than or equal to  $$(m-1)(|\mu|-r(m-2))+\mu_r-m+2.$$
On the  other hand, the degrees  of the  Wronskians are at most $(m-1)(|\mu|-m+2).$ Thus, 
$$(m-1)(|\mu|-m+2)\geq (m-1)(|\mu|-r(m-2))+\mu_r-m+2,$$
which implies 
$-(m-1)(m-2)\geq -r(m-1)(m-2)+\mu_r-m+2.$ 
 After straightforward simplifications the latter inequality gives
$$m-1 \geq \sqrt{\frac{\mu_r}{r-1}}.$$
Contradiction.
\end{proof}

\begin{corollary}
For $\mu=(\mu_1\ge \mu_2 \ge \dots \ge \mu_r)$, 
 either $\ell_\mu\geq \sqrt{\frac{\mu_r}{r-1}}+1$ or any solution of \eqref{eq:SFMU} is a common radical solution.
\end{corollary}

\begin{remark}{\rm 
For any partition $\mu$ with a growing secant degeneracy index, i.e., for $\ell_{\mu^{\ldr}}\rightarrow \infty$,  we know that $$\sqrt{\frac{\mu_r+t}{r-1}}+1\leq \ell_{\mu^{\ldr}} \leq \mu_r+t+2,$$ 
see Proposition~\ref{prop2}  and Theorem~\ref{minjump}. }
\end{remark}

\medskip

\medskip
Now we present a sufficient condition for $\mu$ to have a growing secant degeneracy index.

\begin{corollary}
Any partition $\mu=(\mu_1\ge \mu_2\ge \dots \ge \mu_r)$, such that every its  jump is at least $(r!)^2$, has a growing secant degeneracy index.
\end{corollary}
\begin{proof}
Assume that $\ell_{\mu^{\ldr}}$ does not grow to infinity.
 Then by Theorem~\ref{minjump}, $$\ell_{\mu^{\ldr}}\geq \sqrt{{h_\mu^t}+\frac{1}{4}}+\frac{3}{2}\geq \sqrt{{h_\mu^0}+\frac{1}{4}}+\frac{3}{2}\geq \sqrt{{(r!)^2}+\frac{1}{4}}+\frac{3}{2}>r!.$$
 However the number of different polynomials (up to a constant factor) with fixed $r$ roots of multiplicities $\mu^{\ldr}$ is at most $r!$. Hence  no common radical solution can exist. Contradiction.
\end{proof}

We continue with the proof of Theorem~\ref{prop:parking}. 

\begin{proof}
Let $f=(x-c_1y)^{\mu_1}\cdot(x-c_2y)^{\mu_2}\cdot\ldots\cdot(x-c_ry)^{\mu_r}$ be any form from $S_\mu$. 
Consider the set $\cD_{(a_1,\dots,a_r)}$ of permutations of the multiset $\tau_\mu$ satisfying the  asumptions  (i) and (ii) of Theorem~\ref{prop:parking}.
For any $\pi\in \cD_{(a_1,\dots,a_r)}$, define $f_\pi$ as the form from the $Sym_r$-orbit of $f$ corresponding to $\pi$. 
Any such form $f_\pi$ is divisible by $g=(x-c_1y)^{a_1}\cdot(x-c_2y)^{a_2}\cdot\ldots\cdot(x-c_ry)^{a_r}$, because $\pi$ satisfies the above assumptions. 

For any $\pi \in \cD_{(a_1,\dots,a_r)},$ define  $\hat{f}_\pi:=\frac{f_\pi}{g}\in S_{|\mu|-\sum_{i=1}^r{a_i}}$. If $|\cD_{(a_1,\dots,a_r)}|\geq |\mu|-\sum_{i=1}^r{a_i}+2$, then the forms $\hat{f}_\pi $ are linearly dependent. Therefore, the forms $f_\pi$, where $\pi$ runs over $\cD_{(a_1,\dots,a_r)}$, are also linearly dependent. 
\end{proof}

The next proposition shows that if there are many jumps of small sizes, then the secant degeneracy  index is bounded. 

\begin{proposition}
Let $d$ be a positive integer greater than $45$.
For a partition $\mu=(\mu_1\ge \mu_2 \ge \dots \ge \mu_r)$, if the number of jumps of sizes less than or equal to $d$ is at least $2(\log d+\log \log d+2)$, then 
$\ell_\mu\leq d(\log d+\log \log d )+2$. (Here by $\log$ we mean the binary logarithm, i.e. the logarithm  with base $2$.)
\end{proposition}

\begin{proof}
Assume that  there are at least $2(\log d+\log\log d)+2)$ such jumps.
Consider every second such jump; the number of these jumps is at least $t=[\log d+\log\log d+2]$.
Assume that  they occupy positions $j_1<\ldots <j_t$, i.e. $(\mu_{j_i}-\mu_{j_i+1})\leq d$, for $i\in[1,t]$. Furthermore $j_i+1<j_{i+1}$, because there is a nontrivial jump between them.

Consider the set of  permutations of $\mu=\{\mu_1,\ldots,\mu_r\}$ such that: 
\begin{itemize}
\item  $i\notin \{j_1,\dots,j_t\}$, $\pi_i\geq \mu_i$;
\item  $i\in \{j_1,\dots,j_t\}$, $\pi_i\geq \mu_{i+1}$.
\end{itemize}

The number of such permutations is $2^t$. By Theorem~\ref{prop:parking}, there is a solution of the size $\sum_{i=1}^t(\mu_{j_i}-\mu_{j_i+1})+2\leq d\cdot t + 2$, because
$$d\cdot t+2 \leq d\cdot(\log d+\log \log d+2)+2\leq
 d\cdot (\log d+\log \log d+3 )\leq$$
 $$\leq d\cdot 2\cdot \log d\leq 2^{\log d+\log \log d+1}\leq 2^{t}$$
which finishes the proof.
\end{proof}

\medskip
\subsection{Examples}

It is rather obvious that all partitions with two parts have growing secant degeneracy index. Indeed, if there exists a common radical solution of \eqref{eq:SFGEN}, then its length $m$ is smaller than or equal to $r!$ which in case of two parts  equals $2$.

\begin{proposition}\label{prop:2parts} \rm {(i)} For partitions $\mu=(p,2)$,  one has that $\ell_\mu=3$ when $p=2,$ and $\ell_\mu=4$ when $p>2$. 

\noindent 
\rm {(ii)} For partitions $\mu=(p,3)$, one has that $\ell_\mu=4$ when $p=3,4,5$ and $\ell_\mu=5$ when $p>7$. Cases $\mu=(6,3)$ and $\mu=(7,3)$ are still open.
\end{proposition}

\begin{proof}
In case $\mu=(3,4)$ we found an  example: 
$$y^3(x+y)^4-y^3x^4=L(x+ay)^3y^4+(1-L)(x+by)^3y^4,$$

\noindent where $a=3-\sqrt{3}$, $b=3+\sqrt{3}$ and $L=\frac{9-5\sqrt{3}}{18}$; 

In case $\mu=(5,3)$ we found an example:  
$$f_1+f_2-f_3-f_4=0,$$

\noindent where $f_1(x)=(x+c_1^5y)^3(x+c_1^{-3}y)^5$;  $f_2(x)=(x+c_2^5y)^3(x+c_2^{-3}y)^5$;  $f_3(x)=(x+c_1^{-5}y)^3(x+c_1^{3}y)^5$;  $f_4(x)=(x+c_2^{-5}y)^3(x+c_2^{3}y)^5$. Here $$c_1=-c_2= \left( \frac{1+i\sqrt{35}}{6} \right)^{\frac{1}{4}}.$$  
\end{proof}

 Using computer algebra packages we were able to prove the following statement. 

\begin{proposition} \label{3-parts} For a partition $\mu$ with  three parts $a\le b \le c,$ the following three conditions are equivalent: 

\noindent
{\rm (i)} $\mu$ has a stabilising secant degeneracy index;

\noindent
{\rm (ii)} $\mu$ has a strongly stabilising secant degeneracy index;

\noindent
{\rm (iii)} the triple  $a\le b \le c$ belongs to one of the following three types: $a=b$, $c=a+1$; $b=a+1$, $c=a+2$;   $b=a+1$, $c=a+3$.

\end{proposition}

{
	\begin{proof}
		
		${\rm (i) \Leftrightarrow{}\rm(ii)}$.  Note that any triple of distinct points in $\bC P^1$ can be transformed into any other triple of distinct points in $\bC P^1$ by a M\"obius transformation. 
		Therefore, if a stabilising solution exists for some choice of three distinct roots, it exists for
		any other triple of distinct roots.
		
		${\rm (ii) \Leftrightarrow{}\rm(iii)}$. Consider the $S_3$-orbit of the function $f_1(t)=(t-x_1)^a(t-x_2)^b(t-x_3)^c$ containing 6 pairwise distinct functions
		$f_j$, $j=1,\dots,6$ in case $a<b<c$ and 3 such functions if any two of them coincide. Denote by $W$ the corresponding Wronskian of $f_1,\dots,f_6$ (divided by the trivial factor which is a polynomial of degree 6 in $t$).
		If $\mu=(a,b,c)$ has strongly stabilising secant degeneracy index, then $W\equiv 0$ for all
		$t,x_1,x_2,x_3$. In particular, all coefficients of $W(t)$ should vanish identically.  Using Maple we  able to find  all possible triples $(a,b,c)$ for which all coefficients of $W(t)$ are identically $0$. The solutions are presented
		in {\rm (iii)} above. 
		
	\end{proof}
	
}









\section{ Final remarks and problems} \label{finsec}

\noindent 
{\bf 1.}  Conjecture~\ref{conj:converse}  from  the introduction is equivalent  to the claim that for an appropriate choice of an $r$-tuple of distinct complex numbers, a certain matrix whose entries are polynomials in these numbers  has full rank. In other words, that some maximal minor of this matrix is a non-trivial polynomial  in the latter $r$-tuple which is highly plausible. Unfortunately, the structure of the matrix is rather involved and so far we are only able to handle a large number of  special cases. 

\smallskip\noindent
{\bf 2.} Conditions formulated in  Proposition~\ref{pr:general} are difficult to check which motivates the following questions.  

\medskip
\begin{problem} Give necessary and sufficient combinatorial conditions for a partition $\mu$ to have a growing/stabilizing secant degeneracy index. \end{problem}


\medskip
\begin{problem} For any partition $\mu$ with a growing secant degeneracy index, what is the leading term  of the  asymptotic of $\ell_{\mu^{\ldr}},$ when $t\to +\infty$?
Does it depend on  a particular choice of $\mu$? 
\end{problem}

\medskip
\noindent
{\bf 3.}
The next statement is obvious. 
\begin{proposition}\label{pr:mumubar}
The number $\ell_{\bar \mu}$ is monotone non-decreasing in the refinement partial order. In other words, if $\mu^\prime \succ \mu^{\prime\prime},$ then $\ell_{\bar \mu^\prime}\ge \ell_{\bar\mu^{\prime\prime}}$. \end{proposition}


 Based on a substantial number of calculations, we conjecture the following. 

\begin{conjecture}\label{conj1}
For any partition $\mu\vdash d$, there exists  $\mu^\prime \succeq \mu$ such that $\ell_{\bar \mu}=\ell_{\mu^\prime}$.
\end{conjecture}

\end{document}